\documentclass[11pt]{amsart}

\usepackage{amssymb, amsthm, amsmath}
\usepackage{xcolor}
\newtheorem{theorem}{Theorem}[section]

\newtheorem{lemma}[theorem]{Lemma}
\newtheorem{corollary}[theorem]{Corollary}

\newtheorem*{theorem*}{Theorem}{\bf}{\it}
\newtheorem*{proposition*}{Proposition}{\bf}{\it}
\newtheorem*{observation*}{Observation}{\bf}{\it}
\newtheorem*{lemma*}{Lemma}{\bf}{\it}

\theoremstyle{definition}

\theoremstyle{remark}
\newtheorem{remark}[theorem]{Remark}

\newcommand{\Z}{\mathbb Z}
\newcommand{\R}{\mathbb R}

\def\XXint#1#2#3{{\setbox0=\hbox{$#1{#2#3}{\int}$ }
\vcenter{\hbox{$#2#3$ }}\kern-.6\wd0}}

\begin{document}
\title[]{Local version of Courant's nodal domain theorem}

\author{S. Chanillo}
\address{Sagun Chanillo: Department of Mathematics - Hill Center 
Rutgers, The State University of New Jersey, NJ, USA}
\email{chanillo@math.rutgers.edu}

\author{A. Logunov}
\address{Alexander Logunov: Department of Mathematics, Princeton University, Princeton, NJ, USA}
\email{log239@yandex.ru}

\author{E. Malinnikova}
\address{Eugenia Malinnikova: Department of Mathematics, Stanford University, Stanford, CA, USA}
\email{eugeniam@stanford.edu}

\author{D. Mangoubi}
\address{ Dan Mangoubi:
	Einstein Institute of Mathematics,
	The Hebrew University,
	Jerusalem,
	Israel}
\email{dan.mangoubi@mail.huji.ac.il}
\begin{abstract}
Let $(M^n,g)$ be a closed $n$-dimensional Riemannian manifold, where  $g=(g_{ij})$ is $C^1$-smooth metric.
Consider the sequence of eigenfunctions $u_k$ of the Laplace operator on $M$. Let $B$ be a ball on $M$.
We prove that the number of nodal domains of $u_k$ that intersect  $B$
is not greater than $$C_1\frac{\textup{Volume}_g(B)}{\textup{Volume}_g(M)}k + C_2 k^{\frac{n-1}{n}},$$
where $C_1$, $C_2$ depend on $M$.
The problem of local bounds for the volume and for the number of nodal domains was raised by Donnelly and Fefferman, who also proposed an idea how one can prove such bounds.
 We combine their idea with two ingredients:
the recent sharp Remez type inequality for eigenfunctions and the Landis type growth lemma in narrow domains. 
\end{abstract}
\maketitle
\section{Introduction}
Let $(M^n,g)$ be a closed Riemannian manifold with $C^1$-smooth $g_{ij}$.
The spectrum of the Laplace operator on $M$ is discrete. There is a sequence of 
eigenvalues $$0=\lambda_1 < \lambda_2 \leq \lambda_3 \dots$$
that tend to $\infty$ 
and a sequence of (real) eigenfunctions $u_k$ such that 
$$ \Delta_g u_k + \lambda_k u_k=0.$$
Our enumeration of eigenvalues is non-standard. We start with $\lambda_1=0$ and
$u_1=1$ on $M$. 
The nodal domains of $u_k$ are the connected components of $M\setminus Z_{u_k}$, where 
$Z_{u_k}$ is the zero set of $u_k$ ($Z_{u_k}$ is called the nodal set of $u_k$).
The Courant nodal domain theorem states that the $k$-th eigenfunction $u_k$ has
at most $k$ nodal domains. If the multiplicity of an eigenvalue is more than 1, one may enumerate the eigenfunctions corresponding to this eigenvalue in any order. 
Our main result is the local version of Courant's theorem.

\pagebreak

\begin{theorem} \label{Courant}  Consider a ball $B=\{ x \in M: d_g(x,x_0) < r \}$ with center at  $x_0\in M$ and radius $r<r_0(M)$.  
For any eigenfunction $u_k$, the number of connected components of $B \setminus Z_{u_k} $ that intersect $\frac{1}{2}B$ \textup{(}the local number of nodal domains in $B$\textup{)} is not greater than
$$  C_1\frac{\textup{Volume}_g(B)}{\textup{Volume}_g(M)}k + C_2 k^{\frac{n-1}{n}},$$
where $C_1,C_2$ depend only on $(M,g)$ and are independent of $k$ and $B$. 
\end{theorem}

 The main issue addressed in this paper is why a nodal domain cannot be very long and narrow. A local bound on the volume of a nodal domain is proved in Theorem \ref{main}.
 
The starting point is the idea due to Donnelly and Fefferman \cite{DF90} that one can use growth estimates for eigenfunctions to prove the local bounds for the volume and for the number of nodal domains. Donnelly and Fefferman showed in \cite{DF90} that if $B$ is a ball of radius $1/\sqrt{\lambda_k}$, then the number of connected components of $B\setminus Z_{u_k}$ that intersect a twice smaller ball with the same center is at most
 $C_1(M)k^{C_n}$, where $C_n$ is an explicit constant strictly bigger than 1, which depends on the dimension $n$ only.
Donnelly and Fefferman conjectured  that $C_n$ can be improved to 1 (like in the Courant theorem).
 The constant $C_n$ was improved  by Chanillo and Muckenhoupt  \cite{CM91} and then  by Lu \cite{L93}, and  by Han and
Lu \cite{HL11}.  The arguments involved subtle versions of covering lemmas and improvements to BMO bounds for $\log|u_k|$.   However the improved $C_n$ was still bigger than 1. In this work we don't use the language of BMO norms, and argue in terms of closely related Remez type inequality. We discuss growth estimates of eigenfunctions in Section \ref{sec: growth}.

  In the proof of Theorem \ref{Courant} we combine the idea  of Donnelly and Fefferman with two ingredients. The first ingredient is
the sharp Remez type inequality for eigenfunctions, which is related to a (resolved) conjecture by Landis on a three balls theorem for wild sets. The second ingredient is Landis type growth lemma for eigenfunctions. The lemma and the conjecture by Landis are independent statements. These two ingredients  give two estimates of growth, which however compete with each other.  Our local bound on the number 
  of nodal domains on scale $1/\sqrt \lambda_k$ is $Ck^{\frac{n-1}{n}}$, which
  is better than $Ck$. The bound $Ck$ follows from the idea of Donnelly and Fefferman and the sharp Remez type inequality for eigenfunctions. The improvement from $Ck$ to $Ck^{\frac{n-1}{n}} (\log k)^{n-1}$ is obtained using  the local volume bound for nodal domains proved in Section \ref{sec: volume}. The proof uses one more tool due to Landis, which is discussed in Section \ref{sec: Landis}. This approach is similar to the one in \cite{M}, where estimates on the volume of a connected component of $B_{1/\sqrt{\lambda_k}}\setminus Z(u_k)$ are obtained under the additional assumption that the Riemannian metric is real-analytic.
 Finally, the sharp bound  $Ck^{\frac{n-1}{n}}$  follows  by adding the argument of Fedor Nazarov in Section \ref{sec: Fedya}. This sharp bound is new for spherical harmonics.

 We would like to mention that in dimension two there is an alternative approach due to the elegant idea of A. Eremenko (see Appendix), which however is  two-dimensional only.
  
  Our local bound on the number of nodal domains is sharp. 
When the dimension $n=2$,  one may consider the unit sphere in $\mathbb{R}^3$
and the restriction on $S^2$ of the homogeneous harmonic polynomial $u(x,y,z)=\Re(x+iy)^m$.
It is a spherical harmonic with $\lambda_k\asymp k \asymp m^2$, whose nodal set consists of $m$ circles intersecting at the poles.
Let $x$  be the North pole. Then the number of nodal domains intersecting a geodesic ball $B_r(x)$ on $S^2$ is $\asymp \sqrt k$.  
This example generalizes to higher dimensions, see \cite[Proposition 7.1]{M}.
On the other hand, consider an eigenfunction on the standard $n$-dimensional flat torus,  $T=(\R/\Z)^n$,  of the form \[u(x_1,..,x_n)=\prod_{j=1}^n \sin(2\pi mx_j).\]
	The corresponding eigenvalue is $\lambda_k\asymp k^{2/n}\asymp m^2$. The nodal domains are $m^n$  cubes on the torus. If  $B$ is a ball of radius $r>C/m$, then the number of nodal domains intersecting $B$ is $\asymp (rm)^n\asymp kr^n$.

 \textbf{Notation.} By $B_r(x)$ we will denote a ball with center at $x$ and radius $r$ in local coordinates on $M$. So $B_r(x)$ can be identified with a Euclidean ball and the symbol $|B_r(x)|$ will be used not for the volume with respect to the metric $g$, but for the Euclidean volume of $B_r(x)$ in local coordinates. Note that these two volumes are comparable. By $B$ we will denote a Euclidean ball in local coordinates on $M$, whose radius and center are not specified, and $B_r$ will be used for any ball of radius $r$.  By $r_0,c,c_1,c_2, \dots$ we will denote small constants and by $C,C_1,\dots$ large constants that may depend on $(M,g)$, but are independent of $\lambda$. In local coordinates we always assume that  the operator $\Delta_g$ is uniformly elliptic with bounded derivatives of the coefficients $g_{ij}$.

It is more convenient to argue in terms of $\lambda_k$ rather than in terms of $k$. The Weyl law states that 
$\lambda_k^{n/2} \sim c k$. Theorem \ref{Courant} follows from the next theorem.

\begin{theorem} \label{main}
Let $u$ be a Laplace eigenfunction with eigenvalue $\lambda>2$ on a closed Riemannian manifold $(M^n,g)$, where $g$ is $C^1$ -- smooth.  Given a ball $B_{2r}$ in local coordinates, consider the connected components $\Omega_i\subset B_{2r}$ of $B_{2r}/Z_u$ that intersect $B_r$. Then for all $\Omega_i$, we have 
$$ |\Omega_i| \geq c\min\left(\lambda^{-n/2}, \frac{r^{n}}{\left(\sqrt \lambda \log \lambda\right)^{n-1}}\right).$$
Furthermore the total number of $\Omega_i$ is 
not greater than $$C\lambda^{n/2}r^n + C  \lambda^{\frac{n-1}{2}}.$$ 
\end{theorem}
\begin{remark} The assumption $\lambda>2$ is just to ensure $\log \lambda > 0$. 
\end{remark}
\begin{remark} The local bound on the number of nodal domains is consistent (on the heuristic level) with the upper bound in Yau's conjecture (see \cite{LM},\cite{Yau})
on the volume of nodal sets: 
$$ H^{n-1}(Z_u) \leq C \sqrt \lambda,$$
where $H^{n-1}$ denotes $(n-1)$ dimensional Hausdorff measure. If  Yau's conjecture is true, then one 
may expect that the average number of nodal domains intersecting a ball of radius $1/\sqrt \lambda$ (average with respect to shifting the center of the ball along the manifold) should be constant (see  \cite{LM}, \cite{NPS}).
If we fix $r$ and let $\lambda$ be very large, then the number of nodal domains in $B_r$ should be not more than
$C(r \sqrt \lambda)^n$. The heuristic obstacle in both questions is the same: why there cannot be many nodal domains that are long, twisting and narrow.
\end{remark}

\noindent\textbf{Acknowledgements.} 
We are grateful to Mikhail Sodin and Fedor Nazarov, who read the draft of the text and
helped to improve it. The preliminary version of the text proved a non-sharp local bound for the number of nodal domains.  Fedor Nazarov suggested an argument that removed extra logarithms and made the statement sharp. 

 This work was completed during the time A.L. served as Sloan Fellow and Packard Fellow.
E.M. was partially supported by NSF grant DMS-1956294 and by the Research Council of Norway, Project 275113. D.M. was supported by  ISF grant no. 681/18.

\section{The main tools.}
\subsection{Remez type inequality for eigenfunctions.} \label{sec: growth}

We note that this section is not self-contained. The first main ingredient of this paper is the Remez type inequality for eigenfunctions.  The complete proof of this inequality is contained in the lecture notes \cite{LM19},\cite{LM18} and we decided not to include it here. 

Let $u$ be a Laplace eigenfunction with eigenvalue $\lambda$ on a closed Riemannian manifold $(M^n,g)$, where $g$ is $C^1$ -- smooth in local coordinates. Donnelly and Fefferman \cite{DF88}  proved the following growth estimate for Laplace eigenfunctions (in \cite{DF88} it is formulated for $C^\infty$ -- smooth metrics, but only $C^1$ -- smoothness of $g_{ij}$ is actually needed, see Appendix):
\begin{equation} \label{eq: doubling}
 \sup_{\frac{1}{2}B} |u| \geq c e^{-C\sqrt \lambda }\sup_{B}|u|
\end{equation} 
for any ball $B \subset M$ and the twice smaller ball $\frac{1}{2}B \subset M$ with the same center. 
\begin{remark}
 The recent works \cite{BD18},\cite{DJ18},\cite{DJN} on the distribution of $L^2$ mass of eigenfunctions imply
 that for two-dimensional surfaces with negative curvature, one can improve  $C\sqrt \lambda $ in \eqref{eq: doubling} to $o(\sqrt \lambda)$.
\end{remark}

 In \cite{DF90} Donnelly and Fefferman raised a question on the distribution of values of the eigenfunctions. One of the simplest versions of their question is the following. Let $C$ be a large constant and normalize the eigenfunction $u$ so that $\sup_M |u|=1$.
Can $|u|$ be $e^{-C\sqrt \lambda}$ small on half of $M$ (is it possible that the set where $|u|< e^{-C\sqrt \lambda}$ has measure at least half of $M$)?

The answer (see \cite{LM19},\cite{LM18}) is  that it cannot happen if $C$ is large enough (depending on $(M,g)$ and independent of $\lambda$). This question is related to  the  Landis conjecture (\cite{KL88},p.169), which states that
 if $h$ is a solution to divergence type equation
  $$\text{div}(A\nabla h)=0$$
  in a ball $B$, where $A$ is an elliptic matrix-valued function with smooth coefficients, then 
 the following version of three balls inequality for wild sets holds.
 
Let $K$ be a closed subset of $\frac{1}{2}B$  with positive volume $|K|>0$. There are constants $\alpha \in (0,1)$ and $C>0$, which depend only on $A$ and $|K|$, such that
 if $|h| \leq \varepsilon $ in $K$ and $|h|\leq 1$ in $B$, then 
 $$ |h|\leq C \varepsilon^{\alpha} \quad \text{ in } \frac{1}{2}B.$$
 The proof of the Landis conjecture  is presented in the lecture notes \cite{LM19}.
 
Finally,  the Remez type inequality for eigenfunctions (see \cite{LM19},\cite{LM18})  states that for any ball $B$ in local coordinates on $M$  and any set $E\subset B$ with positive volume, we have
\begin{equation} \label{eq: Remez}
 \sup_E |u| \geq c \left( \frac{c|E|}{|B|} \right)^{C\sqrt \lambda} \sup_{B} |u|,
 \end{equation}
 where $c,C >0$ depend on $M$.

\subsection{Weak maximum principle and a version of Landis' growth lemma.} \label{sec: Landis}
\begin{lemma}[\cite{L63}, p.24] \label{le: small landis}
 Let $h$ be a solution to 
$$  \textup{div}(A\nabla h)= 0 \quad \text{in } B_1(0)\subset \mathbb{R}^n,$$
where $A$ is an elliptic matrix valued function.  Let $r\in(0,1)$.
 Assume that a domain $\Omega\subset B_r(0)$ contains the origin and  $u$ is zero on $\partial \Omega \cap B_r(0)$.
If $$\frac{|\Omega|}{|B_r(0)|} \leq c_0,$$
where $c_0>0$ is a sufficiently small constant depending on the ellipticity constant of $A$,
 then $$|u(0)| \leq \frac{1}{100} \sup_{\Omega\cap B_r(0)} |u|.$$
\end{lemma}
 We note that the maximum principle does not hold for Laplace eigenfunctions and it creates some obstacles.
 However for narrow domains a version of the maximum principle holds.
\begin{corollary} \label{le: maximum} Let $u$ be the Laplace eigenfunction on $M$ with eigenvalue $\lambda$.
Consider a ball $B_{\rho}$ \textup{(}in local coordinates\textup{)}  with center at $x$ and radius $\rho \leq 1 / \sqrt{\lambda}$. Let $\Omega$ be a connected component of $B_{\rho}\setminus Z_u$ that contains $x$.
If $$ \frac{|\Omega |}{|B_{\rho}|} \leq c_0,$$
where $c_0>0$ is sufficiently small, then   
$$|u(x)| \leq \frac{1}{10} \sup_{\Omega} |u|.$$
\end{corollary}
\begin{proof}
We will use the harmonic extension of eigenfunctions.
 Consider the function $h(x,t)=u(x)\exp(\sqrt \lambda t)$ on $M\times \mathbb{R}$ equipped with the Riemannian metric of the product. A direct computation shows that $h$ is harmonic in $M\times \mathbb{R}$. 
 Consider a ball $B$ on $M\times \mathbb{R}$ of radius $\rho$ with center at $(x,0)$.
 Denote $ \Omega\times (-\rho,\rho)$ by $\widetilde \Omega$.
 Note that $$ \frac{|\widetilde \Omega \cap B|}{|B|} \leq Cc_0.$$
 If $c_0$ is small, then by Lemma \ref{le: small landis}
  $$ |u(x)|=|h(x,0)| \leq \frac{1}{100} \sup_{\widetilde \Omega}|h|= \frac{e}{100} \sup_{ \Omega}|u|.$$
\end{proof}

\noindent\textbf{Definition.} We will say that an open set $\Omega$ is $c_0$ -- narrow on scale $\frac{1}{\sqrt \lambda}$
 if $$ \frac{|\Omega \cap B_{1/\sqrt \lambda}|}{|B_{1 / \sqrt{\lambda}}|} \leq c_0.$$
 for any ball $B_{1/\sqrt \lambda}$ of radius $1/\sqrt \lambda$ with center in $\overline{\Omega}$.

 \begin{lemma}[a version of weak maximum principle]
Let $B$ be an open ball in local coordinates on $M$ with radius $r\leq r_0$, and let $u$ be the Laplace eigenfunction on $M$ with eigenvalue $\lambda$. Suppose that an open set $\Omega \subset B$ is  $c_0$ -- narrow on scale $\frac{1}{\sqrt \lambda}$ and $u=0$ on $\partial \Omega\cap B$. If $c_0$ is sufficiently small \textup{(}smallness depends on $M$ and is independent of $\lambda$\textup{)}, then 
$$\sup_\Omega |u| \leq 3 \sup_{\partial \Omega \cap \partial B} |u|.$$
\end{lemma}
\begin{proof}
WLOG, assume $\sup_{\partial \Omega \cap \partial B} |u|=1$.
 Consider the harmonic function $h(x,t)=u(x)\exp(\sqrt \lambda t)$ on $M\times \mathbb{R}$.
 Let $\widetilde \Omega= \Omega\times (-\frac{1}{\sqrt\lambda},\frac{1}{\sqrt\lambda})$ and let $\eta:\mathbb{R} \to [0,1]$ be a positive continuous function
  such that 
  \begin{itemize}
  \item[•] $\eta=1$ on $[-\frac{1}{2\sqrt\lambda},\frac{1}{2\sqrt\lambda}]$,
  \item[•] $\eta(t)=0$ for $t$ with $|t|>\frac{3}{4}\frac{1}{\sqrt\lambda}$.
  \end{itemize}
  Consider a harmonic function $\tilde h$ in $\widetilde \Omega$ such that
   $\tilde h(x,t)= h(x,t) \eta(t)$ on $\partial \widetilde \Omega$. We note that $\tilde h$ vanishes on the top and on the bottom of the boundary of the cylinder $\widetilde \Omega$. The zero sets of eigenfunctions satisfy the exterior cone condition (see Appendix), so $\Omega$ and $\widetilde \Omega$ also satisfy it. Hence the Dirichlet problem in $\widetilde \Omega$ is solvable, $\tilde h $ exists and by the maximum principle  $$\sup_{\widetilde \Omega}|\tilde h| = \sup_{\partial \widetilde \Omega}|\tilde h| = \sup_{\partial \Omega} |h| \sup_{(-\frac{1}{\sqrt\lambda},\frac{1}{\sqrt\lambda})} \exp(\sqrt \lambda t)= e.$$ Let $x_{max}$ be a point in $\overline{\Omega}$ such
   that $|u(x_{max})|=\sup_\Omega |u|$. We may assume $x_{max} \in \Omega$. Denote $(x_{max},0)$ by $\tilde x_{max}$.
   Note that $$ \frac{|\widetilde \Omega \cap B_{\frac{1}{2\sqrt\lambda}}(\tilde x_{max})|}{|B_{\frac{1}{2\sqrt\lambda}}(\tilde x_{max})|} \leq Cc_0$$
 and $h-\tilde h$ is zero on $\partial \widetilde \Omega \cap \{ -\frac{1}{2\sqrt\lambda}< t <\frac{1}{2\sqrt\lambda} \}$. If $c_0$ is sufficiently small, then by Lemma \ref{le: small landis} applied to $h-\tilde h$, we have 
   $$ |h(\tilde x_{max})-\tilde h(\tilde x_{max})| \leq \frac{1}{100} \sup_{\Omega \times \{-\frac{1}{2\sqrt\lambda},\frac{1}{2\sqrt\lambda}\}}|h-\tilde h|.$$
   Since $|\tilde h| \leq e$ in $\widetilde\Omega$, we have
   $$ |h(\tilde x_{max})| \leq e+\frac{e}{100} + \frac{1}{100} \sup_{\Omega \times \{-\frac{1}{2\sqrt\lambda},\frac{1}{2\sqrt\lambda}\}} |h| = e+\frac{e}{100} + \frac{1}{100}e^{1/2}|h(\tilde x_{max})|.$$
    Thus $|u(x_{max})|=|h(\tilde x_{max})|\leq 3$.
   
\end{proof}

\begin{lemma}[a version of Landis growth lemma] \label{le: Landis} 
 Let $u$ be the Laplace eigenfunction on $M$ with eigenvalue $\lambda$. Suppose that an open set $\Omega \subset B_r$ is  $c_0$ -- narrow on scale $\frac{1}{\sqrt \lambda}$ and $u=0$ on $\partial \Omega\cap B_r$. If $\varepsilon$ is  sufficiently small  and 
$$\frac{| \Omega|}{|B_r|} \leq \varepsilon^{n-1},$$
 then  $$ \sup_{B_{r/2}\cap \Omega}|u| \leq e^{-c/\varepsilon} \sup_{\Omega \cap B_r}|u|.$$
\end{lemma}
\begin{proof} Let $x_0$ be the center of $B_r$. By $B_{\rho}$ we will denote the ball with center at $x_0$ and radius $\rho$ (in local coordinates on $M$). 
Define the monotone function $$M(\rho)=\sup_{\Omega \cap B_\rho}|u|.$$
Let $x_\rho \in \partial B_\rho \cap \Omega$ be a point such that $$|u(x_\rho)|= \sup_{\partial B_\rho \cap \Omega}|u|.$$
By the weak maximum principle $$M(\rho) \leq 3|u(x_\rho)|.$$
Define $\tilde \varepsilon= A \varepsilon$, where $A$ is a large constant to be specified later.
We will show that if $\rho + \tilde \varepsilon r< r$ and if
$$ \frac{|\Omega \cap B_{\rho + \tilde \varepsilon r} \setminus B_{\rho- \tilde \varepsilon r}|}{|B_{\tilde \varepsilon r}|} \leq c_0,$$
 then 
 $$  |u(x_\rho)| \leq \frac{1}{10} M(\rho + \tilde \varepsilon r ).$$
 Indeed, if $\tilde \varepsilon r<1/\sqrt\lambda$ we may apply Corollary \ref{le: maximum}  to $B_{\tilde \varepsilon r}(x_\rho)$, and if $\tilde \varepsilon r\geq 1/\sqrt\lambda$ we may use the fact that $\Omega $ is   $c_0$-narrow on scale $1/\sqrt \lambda$ and apply Corollary \ref{le: maximum}  to $B_{1/\sqrt \lambda}(x_\rho)$.
Hence
 $$  M(\rho + \tilde \varepsilon r )\geq 10 |u(x_\rho)| \geq 3 M(\rho).$$
 Consider the numbers $\rho_k= r/2 + k\tilde \varepsilon r$ such that $\rho_k \in (r/2,r)$.
 The total number of such $k$ is $K \asymp 1/\tilde \varepsilon$. If at least half of $\rho_k$ satisfy $  M(\rho_{k+1} ) \geq 3 M(\rho_k)$, then we are done. If for at least half of $k$ we have $  M(\rho_{k+1}) < 3 M(\rho_k)$ (and therefore
 $|\Omega \cap B_{\rho_{k+1}} \setminus B_{\rho_{k-1}}|\geq c_0 \tilde \varepsilon^{n}|B_r|$),
 then 
 $$ |\Omega| \geq c_1 \tilde \varepsilon^n K |B_r| \geq c_2 \tilde \varepsilon^{n-1} |B_r| > \varepsilon^{n-1} |B_r| \quad \text{if } A=\frac{\tilde \varepsilon}{\varepsilon} \gg 1.$$
Thus it cannot happen that for at least half of $k$ we have $  M(\rho_{k+1}) < 3 M(\rho_k)$.
\end{proof}

\section{Proof of Theorem \ref{main}.}
\subsection{Local bound for the volume of nodal domain.} \label{sec: volume}
 Consider  any connected component $\Omega$ of $B_{2r}/Z_{u}$ that intersects $B_r$.
 We may assume that $\Omega$ is $c_0$ -- narrow on scale $\frac{1}{\sqrt \lambda}$, otherwise 
 $|\Omega|>c_1 \lambda^{-n/2}$ and there is nothing to prove.
 Let $x \in B_r \cap \Omega$ and define $\varepsilon>0$ by
 \begin{equation} \label{eq: epsilon}
  |\Omega|=\varepsilon^{n-1}|B_{r/4}(x)|.
 \end{equation}
Our goal is to show that $\varepsilon$ cannot be too small in terms of $\lambda$. 
  For any $\rho \in (r/4,r]$, we have
  $$ \frac{|\Omega \cap B_\rho(x)|}{|B_\rho(x)|} <  \varepsilon^{n-1}.$$
  Since $x\in \Omega$ and $\Omega$ is open 
  $$ \frac{|\Omega \cap B_\rho(x)|}{|B_\rho(x)|} = 1 \text{ for sufficiently small } \rho.$$
  Consider the interval $(r_0, r)$ such that 
  $$ \frac{|\Omega \cap B_\rho(x)|}{|B_\rho(x)|} < \varepsilon^{n-1}$$
  for all $\rho \in (r_0, r)$ and 
  \begin{equation} \label{eq:volume from below}
   \frac{|\Omega \cap B_{r_0}(x)|}{|B_{r_0}(x)|} =  \varepsilon^{n-1}.
  \end{equation}
   Note that $0<r_0\leq r/4$.
  We may assume that 
   $$ 2^{-k_0-1}r < r_0 \leq 2^{-k_0}r$$
   for some integer $k_0 \geq 2$ and $\varepsilon < 1/2$. For every integer $k\in [0,k_0]$, we can apply  growth Lemma \ref{le: Landis} in $B_{2^{-k}r}(x)$ to obtain 
    \begin{equation} \label{eq: multiply}
 \sup_{\Omega \cap B_{2^{-j-1}r}(x)}|u| \leq e^{-\frac{c_2}{\varepsilon}} \sup_{\Omega \cap B_{2^{-j}r}(x)}|u|,
\end{equation}
and by multiplying \eqref{eq: multiply} for $j=0,1,...,k_0-1$, we have
$$
 \sup_{\Omega \cap B_{r_0}(x)}|u| \leq \left(\frac{r_0}{r}\right)^{\frac{c_3}{\varepsilon}} \sup_{\Omega \cap B_r(x)}|u|
$$    
    and therefore 
    \begin{equation}
 \sup_{\Omega \cap B_{r_0}(x)}|u| \leq \left(\frac{r_0}{r}\right)^{\frac{c_3}{\varepsilon}} \sup_{B_r(x)}|u|.
\end{equation} 
  By the Remez-type inequality \eqref{eq: Remez} and by the choice of $r_0$ described by \eqref{eq:volume from below}, we have 
$$
 \sup_{\Omega \cap B_{r_0}(x)}|u| \geq c\left(c \frac{|\Omega \cap B_{r_0}(x)|}{|B_{r}(x)|}\right)^{C\sqrt \lambda} \sup_{B_r(x)}|u| = $$ $$= c^{C\sqrt \lambda +1} \left( \frac{|\Omega \cap B_{r_0}(x)|}{|B_{r_0}(x)|}\right)^{C\sqrt \lambda} \left( \frac{|B_{r_0}(x)|}{|B_{r}(x)|}\right)^{C\sqrt \lambda} \sup_{B_r(x)}|u| \geq
$$ 
$$\geq c_4^{C\sqrt \lambda +1} \varepsilon^{(n-1)C\sqrt \lambda}  \left( \frac{r_0}{r}\right)^{C\sqrt \lambda} \sup_{B_r(x)}|u|. $$
  So 
  $$ \left(\frac{r_0}{r}\right)^{\frac{c_3}{\varepsilon}} \geq c_4^{C\sqrt \lambda +1} \varepsilon^{(n-1)C\sqrt \lambda}  \left( \frac{r_0}{r}\right)^{C\sqrt \lambda}.$$
 Recall that  $\frac{r_0}{r}\leq\frac 1 4$. If $\frac{1}{\varepsilon} \gg \sqrt \lambda$, then
 $$  \left(\frac{r_0}{r}\right)^{\frac{c_3}{2\varepsilon}} < c_4^{C\sqrt \lambda +1} \left( \frac{r_0}{r}\right)^{C\sqrt \lambda}$$
  and therefore
  $$ \left(\frac{1}{4}\right)^{\frac{c_3}{2\varepsilon}} \geq \varepsilon^{(n-1)C\sqrt \lambda},$$
  which yields
  $$  \varepsilon \log  \frac{1}{\varepsilon} \geq \frac{c_4}{\sqrt \lambda}.$$
  The last inequality implies
  $$ \varepsilon \geq c_5 \frac{1}{\sqrt \lambda \log \lambda} \quad \text {if } \lambda >2.$$
  Thus  by \eqref{eq: epsilon}
  $$ |\Omega| \geq c_6 \frac{r^{n}}{\left(\sqrt \lambda \log \lambda\right)^{n-1}}.$$
  
\subsection{Local bound for the number of nodal domains.}\label{sec: Fedya} We follow the argument suggested by Fedor Nazarov. All misprints, bad notation and presentation are on the authors. 

 Denote by $K$ the number of connected components $\Omega_i$ of $B_{2r}\setminus Z_u$ that intersect a twice smaller ball $B_r$ with the same center and  are $c_0$-narrow on scale $\frac{1}{\sqrt \lambda}$. The number of $\Omega_i$, which are not $c_0$-narrow on scale $\frac{1}{\sqrt \lambda}$ (and therefore have volume at least $c_1\lambda^{-n/2}$) is bounded by $C\lambda^{n/2} r^n$.
We want to show that $K$ cannot be too large and we argue by assuming the contrary and arrive to a contradiction. Let $\varepsilon=a/ \sqrt \lambda$ and assume that $K>2^{n+2}/\varepsilon^{n-1}$. The number $a>0$ will be a small constant  depending only on $(M,g)$. 
Each $\Omega_i$ intersects $B_r$, so there is a point $x_i$ in $\Omega_i \cap B_r$.
Since $$\sum |\Omega_i \cap B_r(x_i)| \leq |B_{2r}|,$$ 
there are at least $\frac{3}{4}K$ domains $\Omega_i$ such that $\frac{|\Omega_i \cap B_r(x_i)|}{|B_{2r}|} \leq \frac{4}{K}$ and therefore
\begin{equation} \label{eq:  narrow}
\frac{|\Omega_i \cap B_r(x_i)|}{|B_r(x_i)|} \leq \varepsilon^{n-1}.
\end{equation}
Denote by $S_m$ the set of domains $\Omega_i$, which are
$c_0$-narrow on scale $\frac{1}{\sqrt \lambda}$ and 
\begin{equation} \label{eq:  narrow2}
\frac{|\Omega_i \cap B_{2^{-j}r}(x_i)|}{|B_{2^{-j}r}(x_i)|} \leq \varepsilon^{n-1} \quad \text{for } j\in \{0,1,\dots , m\}.
\end{equation}
The number of elements in $S_0$ satisfies $|S_0|>K-K/4$ by \eqref{eq:  narrow}.
 Fix $m \geq 0$, and consider only $\Omega_i \in S_m$ and forget about the rest of $\Omega_i$. Let $\rho=r2^{-m}$.
Using \eqref{eq:  narrow2} and applying growth Lemma \ref{le: Landis} for $B_{2^{-j}r}(x_i)$ for $j\in \{0,1,\dots , m\}$, we have 
\begin{equation}
\frac{\sup_{\Omega_i \cap B_{\rho/2}(x_i)}|u|}{\sup_{ B_{2r}}|u|} \leq \left(\frac{\rho/2}{2r} \right)^{c/\varepsilon}.   
\end{equation}
By Remez type inequality \eqref{eq: Remez} applied to $\cup (\Omega_i \cap B_{\rho}(x_i)) $, where the union is taken  over $\Omega_i \in S_m$, we have
$$ c\left( c \frac{\sum |\Omega_i \cap B_{\rho/2}(x_i)|}{|B_{2r}|} \right)^{C\sqrt \lambda }\leq \frac{\sup_{\cup (\Omega_i \cap B_{\rho/2}(x_i))}|u|}{\sup_{ B_{2r}}|u|}.$$
So 
$$c\left( c \frac{\sum |\Omega_i \cap B_{\rho/2}(x_i)|}{|B_{2r}|} \right)^{C\sqrt \lambda } \leq \left(\frac{\rho}{4r} \right)^{c/\varepsilon}$$
and therefore
$$c\left( c \frac{\sum |\Omega_i \cap B_{\rho/2}(x_i)|}{|B_{\rho/2}|} \right)^{C\sqrt \lambda } \leq \left(\frac{\rho}{4r} \right)^{c/\varepsilon - Cn\sqrt \lambda}.$$
Since $\varepsilon = a/ \sqrt \lambda \ll 1/\sqrt \lambda$, we have 
$$  \frac{\sum |\Omega_i \cap B_{\rho/2}(x_i)|}{|B_{\rho/2}|}  \leq C_2 \left(\frac{\rho}{4r} \right)^{c_3/a} \leq 4^{-m-2}. $$
So the number of $\Omega_i \in S_m$ with 
$$ \frac{|\Omega_i \cap B_{\rho/2}(x_i)|}{|B_{\rho/2}(x_i)|} > \varepsilon^{n-1} $$
is not greater than $K4^{-m-2}$. So $|S_{m}\setminus S_{m+1}|\leq K4^{-m-2}$. Note that $S_{m+1}\subset S_m$.
 Thus for each $S_m$ the number of $\Omega_i\in S_m$ is at least $$K- \sum_j K4^{-j}  \geq K/2.$$
However 
$$ \frac{|\Omega_i \cap B_\rho(x_i)|}{|B_\rho(x_i)|} = 1 \text{ for sufficiently small } \rho.$$
The number of $\Omega_i$ is finite by the volume bound proved in Section \ref{sec: volume}.
So $S_m$ should be empty for $m$ sufficiently large and the contradiction is obtained.
\section{Appendix.}
In this section we gather several facts that we used in the proof, and which are well-known to specialists, but the assumption on the metric in the references is $C^\infty$-smoothness in place of $C^1$-smoothness (or it is not clear whether it is $C^1$ or $C^\infty$). Unfortunately, the exposition is not self-contained, and the complete proof uses several exterior nontrivial facts.

Let $u$ be a non-zero solution in $B_1(0)$ to the second order linear elliptic equation:
 $$ \sum_{i,j}\partial_i (a_{ij}\partial_j u) + \sum_j b_j \partial_j u + c u=0,$$
 where $(a_{ij})$ is an elliptic matrix with Lipschitz coefficients, and $b_j$, $c$ are bounded.
 
 Garofalo and Lin \cite{GL86},\cite{GL87} proved a powerful monotonicity formula  for solutions of such equations, which implies the three balls inequality:
 \begin{equation} \label{3balls}
 \sup_{B}|u| \leq C \left(\sup_{\frac{1}{2}B}|u|\right)^\alpha \left(\sup_{2B}|u|\right)^{1-\alpha},
 \end{equation}
where $B$ is ball such that $2B \subset B_1(0)$, and $C>0$, $\alpha \in (0,1)$ depend only on the ellipticity constants, on the Lipschitz constant, and on the $L^\infty$ norm of the lower order terms. 
 
 \textbf{Cone condition for nodal sets.}
 Assume that $(a_{ij})$ is the identity matrix at the origin.
 The monotonicity formula of Garofalo and Lin implies (after some work, see \cite{NV}) that $u$ can be  well approximated by a harmonic polynomial near the origin. The story of approximation properties of elliptic equations  at a point with higher order zero   goes back to the works of Bers \cite{B55}, and  Caffarelli and Friedman \cite{CF85}.
The formal statement is that there is a non-zero homogeneous harmonic polynomial $P$ of some degree $n$,
 such that 
 $$ u(x)= P(x)+o(|x|^n).$$
 For homogeneous polynomials one can inscribe a cone  with vertex at $0$ in every nodal domain
  so that  $|P(x)| \geq c|x|^n$, $c>0$, in this cone.
 As the corollary of this fact one may conclude the following.

 Let $u$ be the Laplace eigenfunction on $(M,g)$, where $g=(g_{ij})$ is $C^1$-smooth.
  Let $\Omega$ be a nodal domain of $u$ and a point $x_0\in \partial \Omega$. Then
  there is an open truncated cone in $\Omega$ with vertex at $x_0$.
  It implies the exterior cone condition as well.
We refer to \cite{GM} for the question and estimates (in terms of the eigenvalue) on how large  the spatial angle of the cone is.

 \textbf{Doubling condition for eigenfunctions due to Donnelly and Fefferman.}
 Consider a ball $\widetilde B$ in local coordinates on $M$. Then a version of three
 balls inequality holds for eigenfunctions:
   \begin{equation} \label{3balls for eigenfunctions}
 \sup_{B}|u| \leq C e^{C\sqrt \lambda} \left(\sup_{\frac{1}{2}B}|u|\right)^\alpha \left(\sup_{2B}|u|\right)^{1-\alpha} \quad \text{ if } 2B \subset \widetilde B.
 \end{equation}
 Inequality \eqref{3balls for eigenfunctions} can be reduced to \eqref{3balls} by the harmonic extension, i.e. considering a harmonic function $h(x,t)=u(x)\exp(\sqrt \lambda t)$ on $M\times R$.
 Donnelly and Fefferman proved \eqref{3balls for eigenfunctions} in a different way by Carleman inequalities. They used compactness of $M$ to show that iterations of \eqref{3balls for eigenfunctions}  (see the argument in
\cite{DF88}, page 162, after formula (1.5)) imply 
\begin{equation} 
 \sup_{\frac{1}{2}B} |u| \geq c e^{-C\sqrt \lambda }\sup_{B}|u|
\end{equation} 
for any ball $B$ on $M$.

\textbf{Eremenko's lemma}(\cite{M}). Let $u$ be a harmonic function on the plane. The doubling index of $u$ in a ball $B$ is defined by 
$$N= \log\frac{\sup_{2B}|u|}{\sup_{B}|u|}.$$
Then every nodal domain of $u$ in $2B$ that intersects $B$ has area at least $\frac{c}{N+C}$ and therefore 
the number of nodal domains that intersect $B$ is not greater than $CN+C$.

The proof of Eremenko's lemma is an elegant two-dimensional argument  that is using the fact that $\log|\nabla u|$ is subharmonic. The log-subharmonicity property is not true in higher dimensions.

\end{document}